\documentclass[a4paper,11pt,reqno]{amsart}
\usepackage[foot]{amsaddr}

\usepackage{amssymb}
\usepackage{epsfig}
\usepackage{amsfonts}
\usepackage{amsmath}
\usepackage{euscript}
\usepackage{amscd}
\usepackage{amsthm}
\usepackage{enumitem}
\DeclareMathAlphabet{\mathpzc}{OT1}{pzc}{m}{it}
\usepackage{enumitem}
\usepackage{color}
\usepackage{mathtools}
\usepackage[pagebackref, hypertexnames=false, colorlinks, citecolor=red, linkcolor=blue, urlcolor=red]{hyperref}

\usepackage{marginnote}

\usepackage[normalem]{ulem}

\newcommand{\marginextend}[1]{ \addtolength{\oddsidemargin}{-#1}  \addtolength{\evensidemargin}{-#1}
  \addtolength{\textwidth}{#1}\addtolength{\textwidth}{#1}}
\newcommand{\updownextend}[1]{ \addtolength{\topmargin}{-#1}  \addtolength{\textheight}{#1}
\addtolength{\textheight}{#1}}
\marginextend{1cm}
\updownextend{0cm}
\allowdisplaybreaks[4]

\usepackage{pst-node}
\usepackage{tikz-cd}
\usepackage{mathrsfs}
\DeclareFontFamily{OT1}{pzc}{}
\DeclareFontShape{OT1}{pzc}{m}{it}{<-> s * [1.10] pzcmi7t}{}
\DeclareMathAlphabet{\mathpzc}{OT1}{pzc}{m}{it}

\DeclareSymbolFont{SY}{U}{psy}{m}{n}
\DeclareMathSymbol{\emptyset}{\mathord}{SY}{'306}

\theoremstyle{plain}

\newtheorem{thm}{Theorem}[section]
\newtheorem*{thm*}{Theorem}
\newtheorem{cor}[thm]{Corollary}
\newtheorem{lem}[thm]{Lemma}
\newtheorem{prop}[thm]{Proposition}
\newtheorem{defn}[thm]{Definition}
\newtheorem{rem}[thm]{Remark}

\newtheoremstyle{named}{}{}{\itshape}{}{\bfseries}{.}{.5em}{#1 \thmnote{#3}}
\theoremstyle{named}

\numberwithin{equation}{section}

\def\C{{\mathbb C}}

\def\norm#1{\left\|{#1}\right\|}

\def\ov{\overline}

\def\m{\mathcal}
\def\mb{\mathbb}

\def\w{\widehat }

\def\beq{\begin{eqnarray}}
\def\eeq{\end{eqnarray}}
\def\beqa{\begin{eqnarray*}}
\def\eeqa{\end{eqnarray*}}

\def\del{\partial}

\def\ov{\overline}
\def\bl{\boldsymbol}

\def\sgn{{\rm sgn}}
\def\bl{\boldsymbol}
\def\deg{{\rm deg}}

\newcommand{\be}{\begin{equation}}
\newcommand{\ee}{\end{equation}}
\newcommand{\bea}{\begin{eqnarray}}
\newcommand{\eea}{\end{eqnarray}}
\newcommand{\Bea}{\begin{eqnarray*}}
\newcommand{\Eea}{\end{eqnarray*}}
\newcommand{\inner}[2]{\langle #1,#2 \rangle }%

\newcounter{cnt1}
\newcounter{cnt2}
\newcounter{cnt3}
\newcommand{\blr}{\begin{list}{$($\roman{cnt1}$)$}
 {\usecounter{cnt1} \setlength{\topsep}{0pt}
 \setlength{\itemsep}{0pt}}}
\newcommand{\bla}{\begin{list}{$($\alph{cnt2}$)$}
 {\usecounter{cnt2} \setlength{\topsep}{0pt}
 \setlength{\itemsep}{0pt}}}
\newcommand{\bln}{\begin{list}{$($\arabic{cnt3}$)$}
 {\usecounter{cnt3} \setlength{\topsep}{0pt}
 \setlength{\itemsep}{0pt}}}
\newcommand{\el}{\end{list}}

\DeclareMathOperator  {\tr} {tr}

\title{Toeplitz operators on the Hardy spaces of quotient domains}

\author[Ghosh]{Gargi Ghosh}
\address[Ghosh]{Indian Institute of Science, Bangalore, 560012, India}
\email[Ghosh]{gargighosh@iisc.ac.in}

\thanks{The author was supported by IoE-IISc Fellowship.}
\begin{document}

\maketitle
\begin{abstract} Let $\Omega$ be either the unit polydisc $\mb D^d$ or the unit ball $\mb B_d$ in $\mb C^d$ and $G$ be a finite pseudoreflection group which acts on $\Omega.$ Associated to each one-dimensional representation $\varrho$ of $G,$ we provide a notion of the (weighted) Hardy space $H^2_\varrho(\Omega/G)$ on $\Omega/G.$ Subsequently, we show that each $H^2_\varrho(\Omega/G)$ is isometrically isomorphic to the relative invariant subspace of $H^2(\Omega)$ associated to the representation $\varrho.$ For $\Omega=\mb D^d,$ $G=\mathfrak{S}_d,$ the permutation group on $d$ symbols and $\varrho = $ the sign representation of $\mathfrak{S}_d,$ the Hardy space $H^2_\varrho(\Omega/G)$ coincides to well-known notion of the Hardy space on symmetrized polydisc. We largely use invariant theory of the group $G$ to establish identities involving Toeplitz operators on $H^2(\Omega)$ and $H^2_\varrho(\Omega/G)$ which enable us to study algebraic properties (such as generalized zero product problem, characterization of commuting Toeplitz operators, compactness etc.) of Toeplitz operators on $H^2_\varrho(\Omega/G).$ \end{abstract}

\section{Introduction}
The study of Toeplitz operators on function spaces has a long history. Starting from the work of Brown and Halmos in \cite{MR160136}, it has attracted a lot of attention.  They initiated the study of algebraic properties of Toeplitz operators on the Hardy space $H^2(\mathbb D)$ on the unit disc $\mathbb D.$ Subsequently, a vast literature on Toeplitz operators on various function spaces has been developed. However, the study of Toeplitz operators on the reproducing kernel Hilbert space consisting of holomorphic functions on quotient domains has emerged recently. For example, Toeplitz operators on the Hardy space of the symmetrized bidisc and the symmetrized polydisc have been studied in \cite{MR4266152} and \cite{MR4244837}, respectively. Also in \cite{toepop2}, various algebraic properties of Toeplitz operators on the weighted Bergman space of quotient domain $\Omega/G$ have been studied whenever $G$ is a finite pseudoreflection group and $\Omega$ is a bounded $G$-invariant domain in $\mb C^d$. In this article, we study Toeplitz operators on the Hardy space of the quotient domain $\Omega/G$ in the light of Toeplitz operators on $H^2(\Omega)$.

Now we provide a very brief description of the results which are proved in this article. Let $\Omega$ denote either the unit polydisc $\mb D^d=\{\bl z \in \mb C^d: |z_1|, \ldots , |z_d| < 1\}$ or the open unit ball $\mb B_d$ with respect to the $\ell^2$-norm induced by the standard inner product on $\mb C^d.$ Let $G$ be a finite pseudoreflection group which acts on $\Omega$ by $\sigma \cdot \bl z = \sigma \bl z$ for $\sigma \in G$ and $\bl z \in \Omega.$ Associated to each one-dimensional representation $\varrho$ of $G,$ we define the Hardy space $H_\varrho^2(\Omega/G).$ For $\Omega=\mb D^d,$ $G=\mathfrak{S}_d,$ the permutation group on $d$ symbols and $\varrho = $ the sign representation of $\mathfrak{S}_d,$ the Hardy space $H^2_\varrho(\Omega/G)$ coincides to well-known notion of the Hardy space on symmetrized polydisc, given by Misra, Shyam Roy and Zhang in \cite{MR3043017}. Each $H_\varrho^2(\Omega/G)$ can be thought of a subspace of some $L^2$ space with respect to some measure on $\del(\Omega/G),$ where $\del X$ denote the Shilov boundary of the domain $X.$ We also prove that the relative invariant subspace $R_\varrho^G(H^2(\Omega))$ associated to $\varrho$ is isometrically isomorphic to $H_\varrho^2(\Omega/G).$ Then identities involving Toeplitz operators on $H^2(\Omega)$ and $H^2_\varrho(\bl \theta(\Omega))$ are established which allows us to study algebraic properties of Toeplitz operators on the Hardy space of the quotient domain $\Omega/G.$

This is a preliminary draft and in the subsequent draft, we shall add more results and further  directions. 

\section{The weighted Hardy space} In this section, we provide a notion of the (weighted) Hardy space on the quotient domain $\Omega/G$ using invariant theory and representation theory of the group $G.$ \subsection{Pseudoreflection group} A \emph{pseudoreflection} on $\C^d$ is a linear homomorphism $\sigma: \C^d \rightarrow \C^d$ such that $\sigma$ has finite order in $GL(d,\mb C)$ and the rank of $(I_d - \sigma)$ is 1. A group generated by pseudoreflections is called a pseudoreflection group. For example, any finite cyclic group, the permutation group $\mathfrak{S}_d$ on $d$ symbols, the dihedral groups are pseudoreflection groups \cite{MR2542964}. A pseudoreflection group $G$ acts on $\mb C^d$ by $\sigma \cdot \bl z = \sigma \bl z$ for $\sigma \in G$ and $\bl z \in \mb C^d$ and the group action extends to the set of all complex-valued functions on $\mb C^d$ by \bea\label{action}\sigma( f)(\bl z) =  f({\sigma}^{-1}\cdot \bl z), \,\, \text{~for~} \sigma \in G \text{~and~} \bl z \in \mb C^d.\eea A function $f$ is said to be $G$-invariant if $\sigma(f)=f$ for all $\sigma \in G.$ We denote the ring of all complex polynomials in $d$-variables by $\mb C[z_1,\ldots,z_d]$. The set of all $G$-invariant polynomials, denoted by $\mb C[z_1,\ldots,z_d]^G$, forms a subring and coincides with the relative invariant subspace $R^G_{tr}(\mb C[z_1,\ldots,z_d])$ associated to trivial representation of $G$. Chevalley, Shephard and Todd characterized finite pseudoreflection groups in the following theorem. We abbreviate it as CST theorem for further references.
\begin{thm*}[Chevalley-Shephard-Todd theorem]\cite[ p. 112, Theorem 3]{MR1890629}\label{A}
The invariant ring $\C[z_1,\ldots,z_d]^G$ is equal to $\C[\theta_1,\ldots,\theta_d]$, where $\theta_i$'s are algebraically independent
homogeneous polynomials if and only if $G$ is a finite pseudoreflection group. 
\end{thm*}
The collection of homogeneous polynomials $\{\theta_i\}_{i=1}^d$ is called a homogeneous system of parameters (hsop) or a set of basic polynomials associated to the pseudoreflection group $G$. 
The map ${\bl\theta}: \C^d \rightarrow \C^d$, defined by
\bea\label{theta}
{\bl\theta}(\bl z) = \big(\theta_1(\bl z),\ldots,\theta_d(\bl z)\big),\,\,\bl z\in\C^d
\eea is called a basic polynomial map associated to the group $G.$ 
\begin{prop}\label{properholo}
Let $G$ be a finite pseudoreflection group and $\Omega$ be a $G$-space. For a basic polynomial map $\bl \theta: \mb C^d \to \mb C^d$ associated to the group $G,$
\begin{enumerate}
    \item[{\rm (i)}] $\bl \theta(\Omega)$ is a domain and
    \item[{\rm (ii)}] $\bl \theta : \Omega \to \bl \theta(\Omega)$ is a proper map.
    \item[{\rm (iii)}] The quotient $\Omega / G$ is biholomorphically equivalent to the domain $\bl \theta(\Omega)$
\end{enumerate}
\end{prop}
Proof of (i) and (ii) can be found in \cite[Proposition 1, p.556]{MR3133729}. The remaining part of the Proposition follows from \cite[Proposition 1]{MR807258}, see also \cite[Subsection 3.1.1]{MR4404033}, \cite{MR2542964}.
In virtue of (iii), we work with the domain $\bl \theta(\Omega)$ instead of $\Omega / G.$   Special cases of such quotient domains have been studied in many instances. For example, 

\begin{itemize}
\item $\mb D^d / \mathfrak{S}_d$  is realized as the symmetrized polydisc (denoted by $\mb G_d$) in \cite{MR3906291}, where $\mathfrak{S}_d$ is the permutation group on $d$ symbols. 

\item Rudin's domains are realized as $\mb B_d / G$ ($\mb B_d$ is the unit ball of $\mb C^d$ with respect to $\ell_2$-norm) for a finite pseudoreflection group $G$ \cite{MR667790}.

\item In \cite{bender2020lpregularity}, Bender et al. realized a monomial polyhedron as a quotient domain $\Omega / G$ for $\Omega \subseteq \mb D^d$ and a finite abelian group $G.$

\end{itemize}
\begin{rem}
We note a few relevant properties of a basic polynomial map associated to $G$.
\begin{enumerate}
    \item[$1.$] Although a set of basic polynomials associated to $G$ is not unique but the degrees of $\theta_i$'s are unique for $G$ up to order. 
    \item[$2.$]Let $\bl \theta^\prime : \mb C^d \to \mb C^d$ be another basic polynomial map of $G.$ Then $\bl \theta^\prime(\Omega)$ is biholomorphically equivalent to $\bl \theta(\Omega)$ \cite[p. 5, Proposition 2.2]{kag}. Therefore, the notion of a basic polynomial map can be used unambiguously in the sequel. 
    \item[$3.$] Clearly, any function on $\bl \theta(\Omega)$ can be associated to a $G$-invariant function on $\Omega.$ Conversely, any $G$-invariant function $u$ on $\Omega$ can be written as $u=\widehat{u} \circ \bl \theta$ for a function $\widehat{u}$ on $\bl \theta(\Omega).$ Let $\mathfrak q : \Omega \to \Omega/G$ be the quotient map. Since the function $u$ is $G$-invariant, so $u=u_1\circ\mathfrak q$ for some function $u_1$ defined on $\Omega/G$. It is known that $\bl \theta = \mathfrak h \circ \mathfrak q$ for a biholomorphic map $\mathfrak h : \Omega/G \to \bl \theta(\Omega)$ \cite[p. 8, Proposition 3.4]{MR4404033}. Then $u$ can be written as $u=\widehat{u}\circ \mathfrak h \circ \mathfrak q=\widehat{u} \circ \bl \theta.$   
\end{enumerate}\end{rem}

\subsection{One-dimensional representations} Since the one-dimensional representations of a finite pseudoreflection group $G$ play an important role in our discussion, we elaborate on some relevant results for the same. We denote the subset of all one-dimensional representations of $G$ in $\w{G}$ by $\w{G}_1$.\begin{defn}
A hyperplane $H$ in $\C^d$ is called reflecting if there exists a pseudoreflection in $G$ acting trivially
on $H$.
\end{defn}
  For a pseudoreflection $\sigma \in G,$ define $H_{\sigma} := \ker({\rm id} - \sigma).$ By definition, the subspace $H_{\sigma}$ has dimension $d-1$. Clearly, $\sigma$ fixes the hyperplane $H_{\sigma}$ pointwise. Hence each $H_\sigma$ is a reflecting hyperplane.  By definition, $H_\sigma$ is the zero set of a non-zero homogeneous linear polynomial $L_\sigma$ on $\C^d$, determined up to a non-zero constant multiple, that is, $H_\sigma = \{\bl z\in\C^d: L_\sigma(\bl z) = 0\}$. Moreover, the elements of $G$ acting trivially on a  reflecting hyperplane forms a cyclic subgroup of $G$. 
  
  Let $H_1,\ldots, H_t$ denote the distinct reflecting hyperplanes associated to the group $G$ and  the corresponding cyclic subgroups are $G_1,\ldots, G_t,$ respectively. Suppose $G_i = \langle a_i \rangle$ and the order of each $a_i$ is $m_i$ for $i=1,\ldots,t.$ For every one-dimensional representation $\varrho$ of $G,$ there exists a unique $t$-tuple of non-negative integers $(c_1,\ldots,c_t),$ where $c_i$'s are the least non-negative integers that satisfy the following: \bea\label{ci}\varrho(a_i) =\big( \det(a_i)\big)^{c_i}, \,\, i=1,\ldots,t.\eea The $t$-tuple $(c_1,\ldots,c_t)$ solely depends on the representation $\varrho.$ The character of the one-dimensional representation $\varrho,$ $\chi_\varrho : G \to \mb C^*$ coincides with the representation $\varrho.$ 
 The set of polynomials relative to the representation $\varrho \in \w{G}_1$ is given by \bea\label{invar} R^G_{\varrho}(\mb C[z_1,\ldots,z_d]) = \{f \in \mb C[z_1,\ldots,z_d] : \sigma(f) = \chi_\varrho(\sigma) f, ~ {\rm for ~~ all~} \sigma \in G\}.\eea  The elements of the subspace $R^G_{\varrho}(\mb C[z_1,\ldots,z_d])$ are said to be  $\varrho$-invariant polynomials. Stanley proves a fundamental property of the elements of $R^G_{\varrho}(\mb C[z_1,\ldots,z_d])$ in \cite[p. 139, Theorem 3.1]{MR460484}.
\begin{lem}\label{gencz}\cite[p. 139, Theorem 3.1]{MR460484}
Suppose that the linear polynomial $\ell_i$ is a defining function of $H_i$ for $i=1,\ldots,t.$ The homogeneous polynomial $\ell_\varrho = \prod_{i=1}^t \ell_i^{c_i }$ is a generator of the module $R^G_{\varrho}(\mb C[z_1,\ldots,z_d])$ over the ring $\mb C[z_1,\ldots,z_d]^G,$ where $c_i$'s are unique non-negative integers as described in Equation \eqref{ci}.
\end{lem}
We call $\ell_\varrho$ by \emph{generating polynomial} of $R^G_{\varrho}(\mb C[z_1,\ldots,z_d])$ over $\mb C[z_1,\ldots,z_d]^G.$ It follows that $\sigma (\ell_{\varrho}) = \chi_{\varrho}(\sigma) \ell_{\varrho}.$ We single out a particular one dimensional representation and the associated generating polynomial. The sign representation of a finite pseudoreflection group $G,$ $\sgn : G \to \mb C^*,$ defined by \bea\label{sign} \sgn(\sigma) = (\det(\sigma))^{-1},\eea is given by ${\rm sgn} (a_i) = \big(\det(a_i)\big)^{m_i-1}=(\det(\sigma_i))^{-1}$, $i=1,\ldots,t,$ \cite[p. 139, Remark (1)]{MR460484} and it has the following property. 

\begin{cor}\cite[p. 616, Lemma]{MR117285}\label{Jac}
Let $H_1,\ldots, H_t$ denote the distinct reflecting hyperplanes associated to the group $G$ and let $m_1,\ldots, m_t$ be the orders of the corresponding cyclic subgroups $G_1,\ldots, G_t,$ respectively. Suppose that the linear polynomial $\ell_i$ is a defining function of $H_i$ for $i=1,\ldots,t.$ Then for a non-zero constant $c$, 
\Bea
J_{\bl \theta} (\bl z) = c \prod_{i=1}^t \ell_i^{m_i -1 }(\bl z) = \ell_{\rm sgn}(\bl z),
\Eea
where $J_{\bl \theta}$ is the determinant of the complex Jacobian matrix of the basic polynomial map $\bl \theta.$
Consequently, $J_{\bl \theta}$ is a basis of the module $R^G_{{\rm sgn}}(\mb C[z_1,\ldots,z_d])$ over the ring $\mb C[z_1,\ldots,z_d]^G$.
\end{cor}
Generalizing the notion of a relative invariant subspace, defined in Equation \eqref{invar}, we define the relative invariant subspace of $L_2(\del\Omega)$ associated to a one-dimensional representation $\varrho$ of $G,$ by
\Bea
R^G_{\varrho}(L^2(\del\Omega)) = \{f \in L^2(\del\Omega) : \sigma(f)  = \chi_\varrho(\sigma) f ~ {\rm for ~~ all~} \sigma \in G \}.
\Eea Similarly, the relative invariant subspace of $H^2(\Omega)$ associated to a one-dimensional representation $\varrho$ of $G,$ is defined by
\Bea
R^G_{\varrho}(H^2(\Omega)) = \{f \in H^2(\Omega) : \sigma(f)  = \chi_\varrho(\sigma) f ~ {\rm for ~~ all~} \sigma \in G \}.
\Eea

\subsection{Definition of the Hardy space on quotient domain} 
Let $G$ be a finite pseudoreflection group which acts on $\Omega.$  The basic polynomial map $\bl \theta: \Omega \to \bl \theta(\Omega)$ is a proper holomorphic map.  We define the Hardy space on $\bl \theta(\mb D^d)$ and $\bl \theta(\mb B_d)$ following \cite{MR3043017}.

{\sf For the polydisc:} Let $d\Theta$ be the normalized Lebesgue measure on the Torus $\mb T^d,$ where $\mb T = \{z \in \mb C: |z| =1\}$ is the unit circle. Let $d\Theta_{\varrho,\bl \theta}$ be the measure supported on the Shilov boundary $\del\bl \theta(\mb D^d)$ of $\bl \theta(\mb D^d)$ obtained from the following:
\bea\label{measure} \int_{\del\bl \theta(\mb D^d)} f d\Theta_{\varrho,\bl \theta} = \int_{\mb T^d} f\circ \bl \theta |\ell_\varrho|^2 d\Theta ,\eea where $\ell_\varrho$ is as defined in Lemma \ref{gencz}. The $L^2$ space of $\del\bl \theta(\mb D^d)$ with respect to the measure $d\Theta_{\varrho,\bl \theta}$ is given by \Bea L^2_\varrho(\del\bl \theta(\mb D^d))= \{f : \del\bl \theta(\mb D^d) \to \mb C | \int_{\del\bl \theta(\mb D^d)} |f|^2 d\Theta_{\varrho,\bl \theta} < \infty \}.\Eea
Associated to each one-dimensional representation $\varrho$ of $G,$ we define the weighted Hardy space $H_\varrho^2(\bl \theta(\mb D^d))$ to be the space consisting of holomorphic functions on $\bl \theta(\mb D^d)$ which satisfy
$ {\rm sup}_{0<r<1} \int_{\mb T^d} |f\circ \bl \theta(re^{i\Theta})|^2 |\ell_\varrho(re^{i\Theta})|^2 d\Theta < \infty.$ This is a Hilbert space with the norm 
\bea\label{norm1}\norm{f}^2 =\frac{1}{|G|} {\rm sup}_{0<r<1} \int_{\mb T^d} |f\circ \bl \theta(re^{i\Theta})|^2 |\ell_\varrho(re^{i\Theta})|^2 d\Theta.\eea We call the weighted Hardy space $H^2_\sgn(\bl \theta(\mb D^d))$ associated to the sign representation of $G$ by the Hardy space on $\bl \theta(\mb D^d)$ and denote it by $H^2(\bl \theta(\mb D^d)).$ For the sign representation of the permutation group $\mathfrak{S}_d,$ this notion of the Hardy space coincides with the same in \cite{MR3043017}. In Lemma \ref{equi}, we show that $H^2_\varrho(\bl \theta(\mb D^d))$ is isometrically embedded inside $L^2_\varrho(\del\bl \theta(\mb D^d)).$

{\sf For the unit ball:} Let us define the Hardy space $H^2(\bl \theta(\mb B_d))$ to be the space consisting of holomorphic functions on $\bl \theta(\mb B_d)$ which satisfy
$ {\rm sup}_{0<r<1} \int_{S_d} |f\circ \bl \theta(rt)|^2 |J_{\bl \theta}(rt)|^2 d\sigma(t) < \infty,$ where $d\sigma$ is the normalized Lebesgue measure on the unit sphere $S_d=\{(z_1,\ldots,z_d) \in \mb C^d: \sum_{i=1}^d |z_i|^2 =1\}.$ We set the norm of $f \in H^2(\bl \theta(\mb B_d))$ by 
\bea\label{norm2}\norm{f}^2=\frac{1}{|G|}{\rm sup}_{0<r<1} \int_{S_d} |f\circ \bl \theta(rt)|^2 |J_{\bl \theta}(rt)|^2 d\sigma(t).\eea We generalize this notion of the Hardy space for each one-dimensional representation of $G.$ For $\varrho \in \widehat{G}_1,$ the weighted Hardy space $H^2_\varrho(\bl \theta(\mb B_d))$ is the space consisting of holomorphic functions on $\bl \theta(\mb B_d)$ which satisfy
$ {\rm sup}_{0<r<1} \int_{S_d} |f\circ \bl \theta(rt)|^2 |\ell_\varrho(rt)|^2 d\sigma(t) < \infty,$ and the norm of $f \in H^2_\varrho(\bl \theta(\mb B_d))$ is given by $\norm{f}^2=\frac{1}{|G|}{\rm sup}_{0<r<1} \int_{S_d} |f\circ \bl \theta(rt)|^2 |\ell_\varrho(rt)|^2 d\sigma(t).$ Let $d\sigma_{\varrho,\bl \theta}$ be the measure supported on the Shilov boundary $\del\bl \theta(\mb B_d)$ of $\bl \theta(\mb B_d)$ obtained from the following:
\bea\label{measure1} \int_{\del\bl \theta(\mb B_d)} f d\sigma_{\varrho, \bl \theta} = \int_{S_d} f\circ \bl \theta |\ell_\varrho|^2 d\sigma.\eea The weighted $L^2$ space of $\del\bl \theta(\mb B_d)$ with respect to the measure $d\sigma_{\varrho,\bl \theta}$ is given by \Bea L^2_\varrho(\del\bl \theta(\mb B_d))= \{f : \del\bl \theta(\mb B_d) \to \mb C | \int_{\del\bl \theta(\mb B_d)} |f|^2 d\sigma_{\varrho,\bl \theta} < \infty \}.\Eea In Lemma \ref{equi}, we show that $H^2_\varrho(\bl \theta(\mb B_d))$ is isometrically embedded inside $L^2_\varrho(\del\bl \theta(\mb B_d)).$

\subsection{Isotypic decomposition and projection operators} We consider the natural action (sometimes called regular representation) of $G$ on $L^2(\del\Omega),$ given by $\sigma (f) (z) = f(\sigma^{-1} \cdot z).$ This action is a unitary representation of $G$ (as the weight $\omega$ is $G$-invariant) and consequently the space $L^2(\del\Omega)$ decomposes into isotypic components.

Now we define the projection operator onto the isotypic component associated to an irreducible representation $\varrho \in \widehat{G}$ in the decomposition of the regular representation on $L_\omega^2(\Omega)$ \cite[p. 24, Theorem 4.1]{MR2553682}. For $\varrho \in \widehat{G},$ the linear operator $\mb P_\varrho:L^2(\del\Omega) \to L^2(\del\Omega)$ is defined by
\Bea \mb P_\varrho \phi = \frac{\deg(\varrho)}{|G|}\sum_{\sigma \in G} \chi_\varrho(\sigma^{-1}) ~ \phi \circ \sigma^{-1}, \, \, \, \phi \in L^2(\del\Omega).\Eea

\begin{lem}
For each $\varrho \in \widehat{G}$, the operator $\mb P_\varrho:L^2(\del\Omega) \to L^2(\del\Omega)$ is an orthogonal projection.
\end{lem}

\begin{proof}
An application of Schur's Lemma implies that $\mb P_\varrho^2 = \mb P_\varrho$ \cite[p. 24, Theorem 4.1]{MR2553682}. 
We now show that $\mb P_\varrho$ is self-adjoint. Using change of variables formula, we get that for all $\phi, \psi \in L^2(\del\Omega) \text{~and~} \sigma \in G,$ \bea\label{inv}\inner{\sigma \cdot \phi}{\sigma \cdot \psi} = \inner{\phi}{\psi},\eea where $\inner{\cdot}{\cdot}$ denotes the inner product in $L^2(\del\Omega).$ For $\phi, \psi \in L^2(\del\Omega),$ we have
\Bea \inner{\mb P_\varrho^* \phi}{\psi} = \inner{\phi}{\mb P_\varrho\psi} &=& \inner{\phi}{\frac{\deg(\varrho)}{|G|}\sum_{\sigma \in G} \chi_\varrho(\sigma^{-1}) ~ \psi \circ \sigma^{-1}}\\ &=& \frac{\deg(\varrho)}{|G|}\sum_{\sigma \in G} \chi_\varrho(\sigma) \inner{\phi}{\psi \circ \sigma^{-1}} \\ &=&  \frac{\deg(\varrho)}{|G|}\sum_{\sigma \in G} \chi_\varrho(\sigma) \inner{\phi \circ \sigma}{\psi} \\&=& \inner{\mb P_\varrho \phi}{\psi}, \Eea
where the penultimate equality follows from Equation \eqref{inv}.
\end{proof}
Therefore, $L^2(\del\Omega)$ can be decomposed into an orthogonal direct sum as follows: \bea\label{l2ortho} L^2(\del\Omega) = \oplus_{\varrho \in \widehat{G}} \mb P_\varrho (L^2(\del\Omega)). \eea
\begin{lem}\label{equal1}
Let $G$ be a finite pseudoreflecion group and $\Omega$ be a $G$-invariant domain in $\C^d.$ Then for every $\varrho \in \widehat{G}_1,$ $\mb P_\varrho(L^2(\del\Omega)) = R^G_{\varrho}(L^2(\del\Omega)).$ 
\end{lem}The proof is analogous to that of \cite[Lemma 3.3]{toepop2}. \begin{lem}\label{quo} Let $\varrho \in \w{G}_1$ and $f \in \mb P_\varrho(L^2(\del\Omega)).$ Then $f = \ell_\varrho~\widehat{f},$ where $\ell_\varrho$ is a generating polynomial of $R^G_{\varrho}(\mb C[z_1,\ldots,z_d])$ over the ring $\mb C[z_1,\ldots,z_d]^G$ and $\widehat{f}$ is $G$-invariant.
\end{lem} \begin{rem}
With a very similar approach, it can be shown that also the space $H^2(\Omega)$ admits a decomposition as above. 
The linear map $\mb P_\varrho:H^2(\Omega) \to H^2(\Omega),$ defined by,
\Bea \mb P_\varrho \phi = \frac{\deg(\varrho)}{|G|}\sum_{\sigma \in G} \chi_\varrho(\sigma^{-1}) ~ \phi \circ \sigma^{-1}, \, \, \, \phi \in H^2(\Omega) \Eea is the orthogonal projection onto the isotypic component associated to the irreducible representation $\varrho$ in the decomposition of the regular representation on $H^2(\Omega)$ \cite[p. 24, Theorem 4.1]{MR2553682} \cite[Corollary 4.2]{MR4404033} and
\bea\label{decomp} H^2(\Omega) = \oplus_{\varrho \in \widehat{G}} \mb P_\varrho (H^2(\Omega)). \eea
\end{rem}
\begin{lem}\label{equal2}
Let $G$ be a finite pseudoreflecion group and $\Omega$ be a $G$-invariant domain in $\C^d.$ Then for every $\varrho \in \widehat{G}_1,$ $\mb P_\varrho(H^2(\Omega)) = R^G_{\varrho}(H^2(\Omega)).$ Moreover, if $f \in \mb P_\varrho(H^2(\Omega)),$ then $f = \ell_\varrho~\widehat{f},$ where $\ell_\varrho$ is a generating polynomial of $R^G_{\varrho}(\mb C[z_1,\ldots,z_d])$ over the ring $\mb C[z_1,\ldots,z_d]^G$ and $\widehat{f}$ is a $G$-invariant holomorphic function.
\end{lem}We recall an analytic version of Chevalley-Shephard-Todd theorem which allows us to state a few additional properties of the elements of the Hardy space $H^2(\Omega)$. First note that $\C[z_1,\ldots,z_d]$ is a free $\C[z_1,\ldots,z_d]^G$ module of rank $|G|$ \cite[Theorem 1, p. 110]{MR1890629}. Further, one can choose a  basis of $\C[z_1,\ldots,z_d]$
consisting of homogeneous polynomials. We choose the basis in the following manner. For each $\varrho\in \widehat G$,  $\mathbb{P}_{\varrho}(\C[z_1,\ldots,z_d])$ is a free module over $\C[z_1,\ldots,z_d]^G$ of rank $\deg (\varrho)^2$ \cite[Proposition II.5.3., p.28]{Pan}. Clearly, $\C[z_1,\ldots,z_d] = \oplus_{\varrho \in \widehat{G}} \mb P_\varrho (\C[z_1,\ldots,z_d]),$ where the direct sum is orthogonal direct sum borrowed from the Hilbert space structure of $H^2(\Omega)$. Then we can choose a basis $\{\ell_{\varrho,i}: 1 \leq i \leq \deg(\varrho)^2\}$ of $\mathbb{P}_{\varrho}(\C[z_1,\ldots,z_d])$ over $\C[z_1,\ldots,z_d]^G$ for each $\varrho \in \widehat G$ such that together they form a basis $\{\ell_{\varrho,i}: \varrho \in \widehat{G} \text{~and~} 1 \leq i \leq \deg(\varrho)^2\}$ of $\C[z_1,\ldots,z_d]$ over $\C[z_1,\ldots,z_d]^G$. We will work with such a choice for the rest of discussion for our convenience.
\begin{thm}[Analytic CST]\cite[p. 12, Theorem 3.12]{MR4404033}\label{acst}
Let $G$ be a finite group generated by pseudoreflections on $\mb C^d$ and $\Omega \subseteq \C^d$ be a $G$-space. For every holomorphic function $f$ on $\Omega$, there exist unique $G$-invariant holomorphic functions $\{f_{\varrho,i}: 1 \leq i \leq \deg(\varrho)^2\}_{\varrho \in \widehat{G}}$ such that $$f = \sum_{\varrho \in \widehat{G}} \sum_{i=1}^{\deg(\varrho)^2} f_{\varrho,i} \ell_{\varrho,i}.$$
\end{thm}
\begin{rem}\rm
With such a choice of basis, we have the following:
\begin{enumerate}
    \item For any element $f$ in $H^2(\Omega),$ there exist unique $G$-invariant holomorphic functions $\{f_{\varrho,i}: 1 \leq i \leq \deg(\varrho)^2\}_{\varrho \in \widehat{G}}$ such that $f = \sum_{\varrho \in \widehat{G}} \sum_{i=1}^{\deg(\varrho)^2} f_{\varrho,i} \ell_{\varrho,i}$ and $\mb P_\varrho f = \sum_{i=1}^{\deg(\varrho)^2} f_{\varrho,i} \ell_{\varrho,i}$ for every $\varrho \in \widehat{G}.$
    \item Additionally, for $\varrho \in \widehat{G}_1,$ whenever $\psi = \ell_\varrho~ \widehat{\psi}  \in H^2(\Omega)$ for some $G$-invariant holomorphic function $\widehat{\psi},$ $\psi$ is in $\mb P_\varrho(H^2(\Omega))$ for $\ell_\varrho$ as described in Lemma \ref{gencz} \cite[Lemma 3.1, Remark 3.3]{kag}. 
    \item  For $\varrho \in \widehat{G}_1,$ any other choice of basis of $\mathbb{P}_{\varrho}(\C[z_1,\ldots,z_d])$ as a free module over $\C[z_1,\ldots,z_d]^G$ is a constant multiple of $\ell_\varrho.$
\end{enumerate}
\end{rem}

\subsection{More on the weighted Hardy space:}
\begin{lem}\label{iden}
The spaces $L^2_\varrho(\del\bl \theta(\Omega))$ and $H^2_\varrho(\bl \theta(\Omega))$ are isometrically isomorphic to $R^G_{\varrho}(L^2(\del\Omega))$ and $R^G_{\varrho}(H^2(\Omega)),$ respectively.
\end{lem}
\begin{proof}
 It is easy to see from the definition that the maps $\Gamma_{\varrho} : L^2_\varrho(\del\bl \theta(\Omega))  \to R^G_{\varrho}(L^2(\del\Omega))$ defined by \bea\label{gam}\Gamma_{\varrho}f = \frac{1}{\sqrt{|G|}} \ell_\varrho f \circ \bl \theta,\eea and $\Gamma_{\varrho} : H^2_\varrho(\bl \theta(\Omega))  \to R^G_{\varrho}(H^2(\Omega))$ defined by \bea\label{gamo}\Gamma_{\varrho}f = \frac{1}{\sqrt{|G|}} \ell_\varrho f \circ \bl \theta\eea are isomteries. Let $\phi$ be in $R^G_{\varrho}(L^2(\del\Omega)).$ From Lemma \ref{equal1} and Lemma \ref{quo}, we have that there exists a $\widehat{\phi}$ such that $\phi = \ell_\varrho ~ \widehat{\phi} \circ \bl \theta$ .  We are to show that $\widehat{\phi} \in L^2_{\varrho}(\del\bl \theta(\Omega))$ which follows from the observation that the norm of $\widehat{\phi}$ is equal to the norm of $\phi.$ Analogous arguments as above with Lemma \ref{equal2} proves that also $\Gamma_{\varrho} : H^2_\varrho(\bl \theta(\Omega))  \to R^G_{\varrho}(H^2(\Omega))$ is surjective. Thus both the maps are unitary.  
\end{proof}

\subsubsection{Formula for an orthonormal basis of $H^2_\varrho(\bl \theta(\Omega))$} Let $\mb N_0 = \mb N\cup \{0\}.$ For $\bl m = (m_1,\ldots,m_d) \in \mb N_0^d,$ $\bl z^{\bl m} = \prod_{i=1}^d z_i^{m_i}$ for $\bl z=(z_1,\ldots,z_d) \in \mb C^d.$ Note that $\{c_{\bl m}\bl z^{\bl m} : \bl m \in \mb N_0^d \}$ forms an orthonormal basis of $H^2(\Omega),$ where $\{c_{\bl m}\}$ depends with the choice of $\Omega ~ (\mb D^d \text{ or } \mb B_d).$  We denote
\begin{enumerate}
    \item $\m I_\varrho = \{\bl m \in \mb N_0^d : \mb P_\varrho \bl z^{\bl m} \neq 0 \}$  and 
    \item $e_{\bl m}(\bl \theta(\bl z)) = \sqrt{|G|} c_{\bl m}\frac{\mb P_\varrho \bl z^{\bl m}}{\ell_\varrho(\bl z)},$ for $\bl m \in \m I_\varrho.$
\end{enumerate} From Lemma \ref{iden}, it follows that $\{e_{\bl m}: \bl m \in \m I_\varrho\}$ provides an orthonormal basis of $H_\varrho^2(\bl \theta(\Omega)).$

In particular, for $\Omega = \mb D^d$ and $G = \mathfrak{S}_d,$ one has  \begin{enumerate}
    \item $\m I_\sgn = \{\bl m \in \mb N_0^d : 0<m_1<\cdots<m_d \}$  and 
    \item the Schur functions $e_{\bl m}(\bl \theta(\bl z)) = \sqrt{d!} c_{\bl m}\frac{\bl a_{\bl m}(\bl z)}{\prod_{i<j}(z_i-z_j)},$ $\bl m \in \m I_\varrho,$ where $\bl a_{\bl m}(\bl z) = \det\Big((\!(z_i^{m_j})\!)_{i,j=1}^d\Big)$ and $c_{\bl m} = \frac{1}{\sqrt{d!}}.$ For an explicit description, see \cite{MR3043017}.
\end{enumerate}

\subsubsection{Formula for the reproducing kernel of $H^2_\varrho(\bl \theta(\Omega))$} Let the reproducing kernel of $H^2(\Omega)$ be denoted by $S_\Omega.$ For $\varrho \in \widehat{G}_1,$ $\mb P_\varrho(H^2(\Omega))$ is a closed subspace of $H^2(\Omega)$ and the reproducing kernel $S_\varrho$ of $\mb P_\varrho(H^2(\Omega))$ is given by $$S_\varrho(\bl z, \bl w) =\frac{1}{\ell_\varrho(\bl z) \ov{\ell_\varrho(\bl w)}}\frac{1}{|G|}\sum_{\sigma \in G} \chi_{\varrho}(\sigma^{-1}) S_\Omega(\sigma^{-1} \cdot \bl z, \bl w).$$ Since $\Gamma_{\varrho} : H^2_\varrho(\bl \theta(\Omega))  \to \mb P_{\varrho}(H^2(\Omega))$ is unitary, so the reproducing kernel $S_{\varrho,\bl \theta}$ of $H^2_\varrho(\bl \theta(\Omega))$ is given by \Bea S_{\varrho,\bl \theta} (\bl \theta(\bl z), \bl \theta(\bl w)) &=& |G|\frac{1}{\ell_\varrho(\bl z) \ov{\ell_\varrho(\bl w)}} S_\varrho(\bl z, \bl w)\\ &=&\frac{1}{\ell_\varrho(\bl z) \ov{\ell_\varrho(\bl w)}}\sum_{\sigma \in G} \chi_{\varrho}(\sigma^{-1}) S_\Omega(\sigma^{-1} \cdot \bl z, \bl w).\Eea\begin{rem} For a fixed $\bl w,$ \bea\label{ne}\ov{\ell_\varrho(\bl w)} \Gamma_\varrho(S_{\varrho,\bl \theta} (\cdot, \bl \theta(\bl w)))(\bl z) \nonumber&=& \frac{1}{\sqrt{|G|}} \ell_\varrho(\bl z) \ov{\ell_\varrho(\bl w)} S_{\varrho,\bl \theta} (\bl \theta(\bl z), \bl \theta(\bl w)) \\&=& \sqrt{|G|}S_\varrho(\bl z, \bl w). \eea\end{rem}
\begin{lem}\label{equi}
For every one-dimensional $\varrho \in \widehat{G}_1,$ $H^2_\varrho(\bl \theta(\Omega))$ is isometrically embedded in $L^2_\varrho(\del\bl \theta(\Omega)).$ 
\end{lem}
\begin{proof}
We know that $\Gamma_\varrho$ as defined in Equation \eqref{gam} and in Equation \eqref{gamo} are unitary operators. There exists a natural isometric embedding $i_\varrho : \mb P_\varrho(H^2(\Omega)) \to \mb P_\varrho(L^2(\del\Omega)).$
We note that the following diagram commutes:
\[ \begin{tikzcd}
H^2_\varrho(\bl \theta(\Omega))\arrow{r}{\Gamma^{-1}_\varrho\circ i_\varrho\circ\Gamma_\varrho} \arrow[swap]{d}{\Gamma_\varrho} & L^2_{\varrho}(\del\bl \theta(\Omega))  \arrow{d}{\Gamma_\varrho} \\%
\mb P_\varrho(H^2(\Omega)) \arrow{r}{i_\varrho}& \mb P_\varrho(L^2(\del\Omega))
\end{tikzcd}
\]
Thus by the isometry $\Gamma^{-1}_\varrho\circ i_\varrho\circ\Gamma_\varrho,$ $H^2_\varrho(\bl \theta(\Omega))$ is embedded into $L^2_\varrho(\del\bl \theta(\Omega)).$
\end{proof}
So without going to technicality, $H^2_\varrho(\bl \theta(\Omega))$ can be thought of a closed subspace of $L^2_\varrho(\del\bl \theta(\Omega)).$ We note that the following diagram commutes:
\[ \begin{tikzcd}
L^2_{\varrho}(\del\bl \theta(\Omega)) \arrow{r}{P_\varrho} \arrow[swap]{d}{\Gamma_\varrho} & H^2_\varrho(\bl \theta(\Omega)) \arrow{d}{\Gamma_\varrho} \\%
\mb P_\varrho(L^2(\del\Omega)) \arrow{r}{\widetilde{P}_\varrho}& \mb P_\varrho(H^2(\Omega))
\end{tikzcd}
\]
where $\widetilde{P}_\varrho$ and $P_\varrho$ are the associated orthogonal projections. 
Note that for $f \in L^2_{\varrho}(\del\bl \theta(\Omega)),$ we have
\Bea (\Gamma_\varrho P_\varrho f)(\bl z) &=&  \frac{1}{\sqrt{|G|}}  (P_\varrho f \circ \bl \theta)(\bl z)\ell_\varrho(\bl z) \\
&=&\frac{1}{\sqrt{|G|}}\ell_\varrho(\bl z) \inner{f}{S_{\varrho,\bl \theta}(\cdot, \bl\theta(\bl z))} \\&=& \frac{1}{\sqrt{|G|}}\ell_\varrho(\bl z) \inner{\Gamma_\varrho f}{\Gamma_\varrho\big(S_{\varrho,\bl \theta}(\cdot, \bl\theta(\bl z))\big)}\\ &=& \frac{1}{\sqrt{|G|}} \inner{\Gamma_\varrho f}{\sqrt{|G|}S_\varrho (\cdot,\bl z)} = (\widetilde{P}_\varrho \Gamma_\varrho f)(\bl z),\Eea
where the penultimate equality follows from Equation \eqref{ne}.


\section{Toeplitz Operators}

In this section, we study algebraic properties of Toeplitz operators on the Hardy spaces of quotient domains.

Recall that for $\widetilde{u} \in L^\infty(\del\Omega),$ Toeplitz operator $T_{\widetilde{u}}$ on $H^2(\Omega)$ is given by $(T_{\widetilde{u}}f)(\bl z) = \widetilde{P}(\widetilde{u}f)(\bl z) = \inner{\widetilde{u}f}{S_\Omega(\cdot, \bl z)},$ where  $\widetilde{P} : L^2(\del\Omega) \to H^2(\Omega)$ is the orthogonal projection and $S_\Omega$ is the reproducing kernel of $H^2(\Omega).$ For $f \in \mb P_\varrho(H^2(\Omega)),$ $\widetilde{u}f \in \mb P_\varrho(L^2(\del\Omega)),$ then from the orthogonal decomposition of $L^2(\del\Omega)$ in Equation \eqref{l2ortho}, we get $(T_{\widetilde{u}}f)(\bl z) = \inner{\widetilde{u}f}{S_\Omega(\cdot, \bl z)} = \inner{\widetilde{u}f}{S_\varrho(\cdot, \bl z)}=\widetilde{P}_\varrho(\widetilde{u}f)(\bl z).$  Therefore, the subspace $\mb P_\varrho(H^2(\Omega))$ remains invariant under $T_{\widetilde{u}}$ and the restriction operator $T_{\widetilde{u}}: \mb P_\varrho(H^2(\Omega)) \to \mb P_\varrho(H^2(\Omega))$ is given by $(T_{\widetilde{u}}f)= \widetilde{P}_\varrho(\widetilde{u}f).$ Moreover, the orthogonal complement of $\mb P_\varrho(H^2(\Omega))$ is $\oplus_{\varrho' \in \widehat{G}, \varrho' \not\equiv \varrho} \mb P_{\varrho'} (H^2(\Omega))$ from Equation \eqref{decomp}. It follows that $\mb P_{\varrho} (H^2(\Omega))$ is a reducing subspace for $T_{\widetilde{u}}.$ 

\begin{lem}\label{equi1}
Let $  \widetilde{u} \in L^\infty(\del\Omega)$ be a $G$-invariant function such that $\widetilde{u}= u\circ \bl \theta.$ For every $\varrho \in \widehat{G}_1,$ the following diagram commutes:
\[ \begin{tikzcd}
H^2_\varrho(\bl \theta(\Omega)) \arrow{r}{T_u} \arrow[swap]{d}{\Gamma_\varrho} & H^2_\varrho(\bl \theta(\Omega)) \arrow{d}{\Gamma_\varrho} \\%
\mb P_\varrho(H^2(\Omega)) \arrow{r}{T_{\widetilde{u}}}& \mb P_\varrho(H^2(\Omega))
\end{tikzcd}
\] where $\Gamma_\varrho$'s are as defined in Equation \eqref{gam} and in Equation \eqref{gamo}.
\end{lem}
\begin{proof}
Note that $\Gamma_\varrho( uf) = \frac{1}{\sqrt{|G|}}  (u \circ \bl \theta)(f \circ \bl \theta)~\ell_\varrho = \widetilde{u}~ \Gamma_\varrho(f).$ From Lemma \ref{equi}, we have
\Bea \Gamma_\varrho T_uf= \Gamma_\varrho(P_\varrho(uf)) = \widetilde{P}_\varrho (\Gamma_\varrho( uf)) =  \widetilde{P}_\varrho (\widetilde{u} \Gamma_\varrho( f)) = T_{\widetilde{u}}\Gamma_\varrho f.\Eea
\end{proof}
Let $\tr:G \to \mb C^*$ denote the trivial representation of the group $G.$ If $\widetilde{u}$ is holomorphic, it is easy to see that the subspace $\ell_\varrho \cdot \mb P_{\tr}(H^2(\Omega))$ is invariant under the restriction of the operator $T_{\widetilde{u}}$ on $\mb P_\varrho(H^2(\Omega)).$
We prove that this continues to hold even when $\widetilde{u}$ is only a bounded function.

Let $f \in \mb P_{\tr}(H^2(\Omega)),$ then $ \ell_\varrho f \in \ell_\varrho \cdot \mb P_{\tr}(H^2(\Omega)) \subseteq R^G_\varrho(H^2(\Omega)) = \mb P_\varrho(H^2(\Omega)),$ where the last equality follows from \cite[Lemma 3.1]{kag}. The density of $G$-invariant polynomials in $\mb P_{\tr}(H^2(\Omega))$ implies that $\ell_\varrho \cdot \mb P_{\tr}(H^2(\Omega))$ is dense in $\mb P_\varrho(H^2(\Omega)).$

We consider $f \in \mb P_\varrho(H^2(\Omega))$ such that $f = \ell_\varrho f_\varrho$ for $f_\varrho \in \mb P_{\tr}(H^2(\Omega)).$  Then $\widetilde{u} f_\varrho \in \mb P_{\rm tr}(L_\omega^2(\Omega))$ using Lemma \ref{equal1} and the following holds: \Bea (T_{\widetilde{u}}f)(\bl z) = \inner{\widetilde{u} f}{S_\Omega(\cdot, \bl z)} = \inner{\widetilde{u} f_\varrho}{M^*_{\ell_\varrho} S_\Omega(\cdot, \bl z)} &=& \ell_\varrho(\bl z) \inner{\widetilde{u} f_\varrho}{S_\Omega(\cdot, \bl z)}  \\ &=& \ell_\varrho(\bl z) \inner{\widetilde{u} f_\varrho}{S_{\tr}(\cdot, \bl z)}  \\ &=& \ell_\varrho(\bl z) \widetilde{P}_{\tr}(\widetilde{u} f_\varrho) (\bl z), \Eea where $S_{\tr}$ denotes the reproducing kernel of $\mb P_{\rm tr}(H^2(\Omega)).$ Therefore, we have that $T_{\widetilde{u}}(\ell_\varrho \cdot \mb P_{\tr}(H^2(\Omega)) \subseteq \ell_\varrho \cdot \mb P_{\tr}(H^2(\Omega))$ 
for every $\varrho \in \widehat{G}_1.$

This result can be extended to any representation $\varrho \in \widehat{G}$ with $\deg( \varrho) > 1.$ 
We consider a basis $\{\ell_{\varrho,i}\}_{i=1}^{\deg( \varrho)^2}$ of $\mb P_\varrho (\mb C[z_1,\ldots,z_d])$ as a free module over $\mb C[z_1,\ldots,z_d]^G.$ Since $\sum_{i=1}^{\deg( \varrho)^2} \ell_{\varrho,i} \cdot \mb C[z_1,\ldots,z_d]^G $ is dense in $\sum_{i=1}^{\deg( \varrho)^2} \ell_{\varrho,i} \cdot \mb P_{\rm tr} (H^2(\Omega))$ and $\sum_{i=1}^{\deg( \varrho)^2} \ell_{\varrho,i} \cdot \mb P_{\rm tr} (H^2(\Omega))$ is contained in $\mb P_\varrho (H^2(\Omega)),$ we get that $\sum_{i=1}^{\deg( \varrho)^2} \ell_{\varrho,i} \cdot \mb P_{\rm tr} (H^2(\Omega))$ is dense in $\mb P_\varrho (H^2(\Omega)).$ For $f = \sum_{i=1}^{\deg( \varrho)^2} \ell_{\varrho,i} f_{\varrho,i},$ such that  $f_{\varrho,i} \in \mb P_{\rm tr} (H^2(\Omega)),$ we conclude the following:
\bea\label{sum}
\nonumber(T_{\widetilde{u}}f)(\bl z) = \inner{\widetilde{u} f}{S_\Omega(\cdot, \bl z)} &=& \inner{\sum_{i=1}^{\deg( \varrho)^2} \ell_{\varrho,i} \widetilde{u} f_{\varrho,i}}{S_\Omega(\cdot, \bl z)} \\&=&\nonumber \sum_{i=1}^{\deg( \varrho)^2} \inner{\widetilde{u} f_{\varrho,i}}{M^*_{\ell_{\varrho,i}}S_\Omega(\cdot, \bl z)} \\&=& \nonumber \sum_{i=1}^{\deg( \varrho)^2} \ell_{\varrho,i}(\bl z) \inner{\widetilde{u} f_{\varrho,i}}{S_\Omega(\cdot, \bl z)} \\ &=&\nonumber \sum_{i=1}^{\deg( \varrho)^2} \ell_{\varrho,i}(\bl z) \inner{\widetilde{u} f_{\varrho,i}}{S_{\tr}(\cdot, \bl z)} \\ &=& \sum_{i=1}^{\deg( \varrho)^2} \ell_{\varrho,i}(\bl z) \widetilde{P}_{\tr}(\widetilde{u} f_{\varrho,i})(\bl z).
\eea
Hence, each $\sum_{i=1}^{\deg( \varrho)^2} \ell_{\varrho,i} \cdot \mb P_{\rm tr} (H^2(\Omega))$ remains invariant for the operator $T_{\widetilde{u}}.$


\begin{thm}
 Suppose that $G$ is a finite pseudoreflection group, the bounded domain $\Omega \subseteq \mb C^d$ is a $G$-space and $\bl \theta: \Omega \to \bl \theta(\Omega)$ is a basic polynomial map associated to the group $G.$ Let $\widetilde{u},\widetilde{v}$ and $\widetilde{q}$ be $G$-invariant functions in $L^\infty(\del\Omega)$ such that $\widetilde{u} = u \circ \bl \theta$, $\widetilde{v} = v \circ \bl \theta$ and $\widetilde{q} = q \circ \bl \theta$. 
 \begin{enumerate}
     \item[\rm 1a.] Suppose that for a one-dimensional representation $\mu$ of $G,$ $T_uT_v=T_q$ on $H^2_\mu(\bl \theta(\Omega)),$ then \begin{enumerate}
    \item[{\rm (i)}] $T_uT_v=T_q$ on $H^2_\varrho(\bl \theta(\Omega))$ for every one dimensional representation $\varrho$ of $G,$ and
    \item[{\rm (ii)}] $T_{\widetilde{u}}T_{\widetilde{v}}=T_{\widetilde{q}}$ on $H^2(\Omega).$
\end{enumerate}
\item[\rm 1b.] Conversely, if $T_{\widetilde{u}}T_{\widetilde{v}}=T_{\widetilde{q}}$ on $H^2(\Omega),$ then $T_uT_v=T_q$ on $H^2_\varrho(\bl \theta(\Omega))$ for every one dimensional representation $\varrho.$
\item[\rm 2a.] Suppose that for a one-dimensional representation $\mu$ of $G,$ $T_uT_v=T_vT_u$ on $H^2_\mu(\bl \theta(\Omega)),$ then \begin{enumerate}
    \item[{\rm (i)}] $T_uT_v=T_vT_u$ on $H^2_\varrho(\bl \theta(\Omega))$ for every one dimensional representation $\varrho$ of $G,$ and
    \item[{\rm (ii)}] $T_{\widetilde{u}}T_{\widetilde{v}}=T_{\widetilde{v}}T_{\widetilde{u}}$ on $H^2(\Omega).$
\end{enumerate}
\item[\rm 2b.] Conversely, if $T_{\widetilde{u}}T_{\widetilde{v}}=T_{\widetilde{v}}T_{\widetilde{u}}$ on $H^2(\Omega),$ then $T_uT_v=T_vT_u$ on $H^2_\varrho(\bl \theta(\Omega))$ for every one-dimensional representation $\varrho.$
\item[\rm 3a.] Suppose that for a one-dimensional representation $\mu$ of $G,$ $T_u$ is compact on $H^2_\mu(\bl \theta(\Omega)),$ then 
\begin{enumerate}
    \item[{\rm (i)}] $T_u$ is compact on $H^2_\varrho(\bl \theta(\Omega))$ for every one-dimensional representation $\varrho$ of $G,$ and
    \item[{\rm (ii)}] $T_{\widetilde{u}}$ is compact on $H^2(\Omega).$
\end{enumerate}
\item[\rm 3b.] Conversely, if $T_{\widetilde{u}}$ is compact on $H^2(\Omega),$ then $T_u$ is compact on $H^2_\varrho(\bl \theta(\Omega))$ for every one dimensional representation $\varrho.$
 \end{enumerate}
\end{thm} \begin{proof}
Proof of 1a - 2b are analogous to that of \cite[Theorem 1.1, Theorem 1.6]{toepop2}, so a detailed proof will be added in future version.

We note that 3b is straightforward from Lemma \ref{equi}. We prove 3a now. Let $\{f_n\}_{n\in \mb N}$ be a bounded sequence in $H^2(\Omega)$ such that $\norm{f_n} <c,\,n\in \mb N,$ for a positive real $c.$ We are to show that $\{T_uf_n\}$ has a convergent subsequence in $H^2(\Omega)$. From analytic version of CST, we write \Bea f_n = \sum_{\varrho \in \widehat{G}} \sum_{i=1}^{\deg(\varrho)^2} f_{n,\varrho,i} \ell_{\varrho,i}.\Eea  Since for two positive real numbers $c_{\varrho,i}$ and $C_{\varrho,i},$ $c_{\varrho,i} < |\ell_{\varrho,i}(\bl z)| < C_{\varrho,i}$ on $\del\Omega \setminus Z(\ell_{\varrho,i}),$ we get $c_{\varrho,i}\norm{h} <\norm{h\ell_{\varrho,i}} < C_{\varrho,i} \norm{h},$ for any $h \in \mb P_{\tr}(H^2(\Omega)).$ Thus we note that $f_{n,\varrho,i} \in \mb P_{\tr}(H^2(\Omega))$ and $\norm{f_{n,\varrho,i}} < k_{\varrho,i}$ for some positive real number $k_{\varrho,i}.$ Then $\{\ell_\mu f_{n,\varrho,i} : n \in \mb N\}$ is bounded for each $\varrho$ and $i.$ So there exists a convergent subsequence $\{\ell_\mu T_{\widetilde{u}}f_{n_{k},\varrho,i} : k \in \mb N\}$ for each $\varrho$ and $i.$ Clearly, $\{\sum_{\varrho \in \widehat{G}} \sum_{i=1}^{\deg(\varrho)^2}\ell_\varrho T_{\widetilde{u}}f_{n_{k},\varrho,i} : k \in \mb N\}$ is a convergent sequence. Since $T_{\widetilde{u}}(\sum_{\varrho \in \widehat{G}} \sum_{i=1}^{\deg(\varrho)^2}\ell_\varrho f_{n_{k},\varrho,i}) = \sum_{\varrho \in \widehat{G}} \sum_{i=1}^{\deg(\varrho)^2}\ell_\varrho T_{\widetilde{u}}f_{n_{k},\varrho,i},$ $T_{\widetilde{u}}$ is compact on $H^2(\Omega).$ 

So the restriction operator $T_{\widetilde{u}}$ on $\mb P_\varrho (H^2(\Omega))$ is also compact and using Lemma \ref{equi1} we conclude 3a (i).

\end{proof}

\begin{cor}
For every $\varrho \in \widehat{G},$ $\mb P_\varrho H^2(\Omega) = \sum_{i=1}^{(\deg \varrho)^2} \ell_{\varrho, i} \cdot \mb P_{\tr} H^2(\Omega).$
\end{cor}
\begin{thm}
    If $T$ is a Toeplitz operator on $H^2_\varrho(\bl \theta(\mb D^d))$ for a one-dimensional representation $\varrho,$ then $M_{z_i}^*TM_{z_i} =T,$ where $M_{z_i}$ denotes the coordinate multiplication on $H^2_\varrho(\bl \theta(\mb D^d))$ by the $i$-th coordinate function of $\bl \theta(\mb D^d).$ 
\end{thm} \begin{proof}
It follows from \cite[Theorem 3.1]{MR3812709} and Lemma \ref{equi1}.
\end{proof}

\bibliographystyle{siam}
\bibliography{Bibliography.bib}

\begin{thebibliography}{10}

\bibitem{MR807258}
{\sc E.~Bedford and J.~Dadok}, {\em Proper holomorphic mappings and real
  reflection groups}, J. Reine Angew. Math., 361 (1985), p.~162–173.

\bibitem{bender2020lpregularity}
{\sc C.~Bender, D.~Chakrabarti, L.~Edholm, and M.~Mainkar}, {\em
  L$^p$-regularity of the {B}ergman projection on quotient domains}, Canadian
  Journal of Mathematics,  (2021), p.~1–37.

\bibitem{MR4266152}
{\sc T.~Bhattacharyya, B.~K. Das, and H.~Sau}, {\em Toeplitz operators on the
  symmetrized bidisc}, Int. Math. Res. Not. IMRN,  (2021), pp.~8492--8520.

\bibitem{MR4404033}
{\sc S.~Biswas, S.~Datta, G.~Ghosh, and S.~Shyam~Roy}, {\em Reducing submodules
  of {H}ilbert modules and {C}hevalley-{S}hephard-{T}odd theorem}, Adv. Math.,
  403 (2022).

\bibitem{MR3906291}
{\sc S.~Biswas, G.~Ghosh, G.~Misra, and S.~Shyam~Roy}, {\em On reducing
  submodules of {H}ilbert modules with {$\mathfrak{S}_n$}-invariant kernels},
  J. Funct. Anal., 276 (2019), p.~751–784.

\bibitem{MR1890629}
{\sc N.~Bourbaki}, {\em Lie groups and {L}ie algebras. {C}hapters 4–6},
  Elements of Mathematics (Berlin), Springer-Verlag, Berlin, 2002.
\newblock Translated from the 1968 French original by Andrew Pressley.

\bibitem{MR160136}
{\sc A.~Brown and P.~R. Halmos}, {\em Algebraic properties of {T}oeplitz
  operators}, J. Reine Angew. Math., 213 (1963/64), pp.~89--102.

\bibitem{MR4244837}
{\sc B.~K. Das and H.~Sau}, {\em Algebraic properties of {T}oeplitz operators
  on the symmetrized polydisk}, Complex Anal. Oper. Theory, 15 (2021),
  pp.~Paper No. 60, 28.

\bibitem{kag}
{\sc G.~Ghosh}, {\em The weighted {B}ergman spaces and pseudoreflection
  groups}, https://arxiv.org/abs/2104.14162v3,  (2021).

\bibitem{toepop2}
{\sc G.~Ghosh and E.~K. Narayanan}, {\em Toeplitz operators on the {B}ergman
  spaces of quotient domains}, https://arxiv.org/abs/2202.03184,  (2022).

\bibitem{MR2553682}
{\sc Y.~Kosmann-Schwarzbach}, {\em Groups and symmetries}, Universitext,
  Springer, New York, 2010.
\newblock From finite groups to Lie groups, Translated from the 2006 French 2nd
  edition by Stephanie Frank Singer.

\bibitem{MR2542964}
{\sc G.~I. Lehrer and D.~E. Taylor}, {\em Unitary reflection groups}, vol.~20
  of Australian Mathematical Society Lecture Series, Cambridge University
  Press, Cambridge, 2009.

\bibitem{MR3812709}
{\sc A.~Maji, J.~Sarkar, and S.~Sarkar}, {\em Toeplitz and asymptotic
  {T}oeplitz operators on {$H^2(\Bbb D^n)$}}, Bull. Sci. Math., 146 (2018),
  pp.~33--49.

\bibitem{MR3043017}
{\sc G.~Misra, S.~Shyam~Roy, and G.~Zhang}, {\em Reproducing kernel for a class
  of weighted {B}ergman spaces on the symmetrized polydisc}, Proc. Amer. Math.
  Soc., 141 (2013), p.~2361–2370.

\bibitem{Pan}
{\sc D.~I. Panyushev}, {\em Lectures on representations of finite groups and
  invariant theory}, https://users.mccme.ru/panyush/notes.html,  (2006).

\bibitem{MR667790}
{\sc W.~Rudin}, {\em Proper holomorphic maps and finite reflection groups},
  Indiana Univ. Math. J., 31 (1982), p.~701–720.

\bibitem{MR460484}
{\sc R.~P. Stanley}, {\em Relative invariants of finite groups generated by
  pseudoreflections}, J. Algebra, 49 (1977), p.~134–148.

\bibitem{MR117285}
{\sc R.~Steinberg}, {\em Invariants of finite reflection groups}, Canadian J.
  Math., 12 (1960), p.~616–618.

\bibitem{MR3133729}
{\sc M.~Trybula}, {\em Proper holomorphic mappings, {B}ell's formula, and the
  {L}u {Q}i-{K}eng problem on the tetrablock}, Arch. Math. (Basel), 101 (2013),
  p.~549–558.

\end{thebibliography}

\end{document}